\documentclass[a4paper,12pt]{amsart}

\RequirePackage[OT1]{fontenc}
\RequirePackage{amsthm,amsmath}
\RequirePackage[colorlinks]{hyperref}

\numberwithin{equation}{section}

\usepackage[T1]{fontenc}
\usepackage[latin2]{inputenc}
\usepackage{indentfirst}
\usepackage[a4paper,left=2.5cm,right=2.5cm,top=2.5cm,bottom=2.5cm,bindingoffset=0cm]{geometry}
\usepackage{amsmath}
\usepackage{amssymb}
\usepackage{amsthm}
\usepackage{dsfont}
\usepackage{comment}
\usepackage{appendix}
\usepackage{graphicx}
\usepackage{subfigure}
\usepackage{caption}
\usepackage{enumerate}
\usepackage{pdfsync}
\usepackage{mathabx}

\newtheorem{thm}{Theorem}[section]    	
\newtheorem{prop}[thm]{Proposition}      	
\newtheorem{lemma}[thm]{Lemma}				
\newtheorem{cor}[thm]{Corollary}		
\newtheorem{defi}[thm]{Definition}

\newtheorem{fact}[thm]{Fact}

\newcommand{\matr}[1]{\underline{\underline{#1}}}

\newcommand*{\be}{\begin{equation}}
\newcommand*{\ee}{\end{equation}}
\newcommand*{\ba}{\begin{aligned}}
\newcommand*{\ea}{\end{aligned}}
\newcommand*{\barr}{\begin{array}{c}}
\newcommand*{\earr}{\end{array}}
\newcommand*{\Ev}{{\mathbb{ E}}}
\newcommand*{\Pv}{{\mathbb{ P}}}

\newcommand*{\ind}{\mathds{1}}

\begin{document}

\title{Projections of Mandelbrot percolation in higher dimensions}

\author{K\'aroly Simon}
\address{Department of Stochastics, Institute of Mathematics, Technical University of Budapest, 1521
Budapest, P.O.Box 91, Hungary} \email{simonk@math.bme.hu} \thanks{K. Simon was supported by OTKA $\#$ K 104745}

\author{Lajos Vágó}
\address{Department of Stochastics, Institute of Mathematics, Technical University of Budapest, 1521
Budapest, P.O.Box 91, Hungary} \email{vagolala@math.bme.hu} \thanks{L. Vagó was supported by TAMOP $\#$ 4.2.2.B-10/1-2010-0009}

\date{\today}

\begin{abstract}

We consider fractal percolation (or Mandelbrot percolation) which is one of the most well studied example of random Cantor sets. Rams and the first author \cite{RS} studied the projections (orthogonal, radial and co-radial) of fractal percolation sets on the plane.
We extend their results to higher dimension.
\end{abstract}

\maketitle

\thispagestyle{empty}

\vspace{-0.7cm}

\section{Introduction}

Fractal percolation, or Mandelbrot percolation \cite{M} (in general sense) on the plane is defined in the following way: Fix an integer $M\geq 2$ and probabilities $0<p_{i,j}<1$, $i,j=1,\dots,M$. Then partition the unit square $K=[0,1]^2$ into $M^2$ congruent squares of side length $1/M$. Let us denote them  by $K_{i,j}$, $i,j=1,\dots,M$. Then retain all small squares $K_{i,j}$ with probability $p_{i,j}$ independently from each other, or discard them otherwise. Repeat this procedure independently in the retained squares ad infinitum to finally get a random set $E$ called fractal percolation. The $d$-dimensional fractal percolation is defined analogously to the two-dimensional one.

The pioneering paper of Marstrand \cite{Mar} asserts that on the plane,
for any Borel set $A$ having Hausdorff dimension greater than one, the orthogonal projection of $A$ to a Lebesgue typical line has positive one-dimensional Lebesgue measure. In \cite{RS} the authors proved that in the case of fractal percolation sets, we can replace "almost all direction" with all direction and "positive one dimensional Lebesgue measure" of the projection of $A$ can be replaced with the existence of interval in the projection.  Mattila \cite{Mattila1} extended Marstrand theorem to higher dimension. Analogously we extend here the result of  \cite{RS} to dimensions higher than two.
 The major difficulty
 caused by the higher dimensional settings is handled
 in Lemma \ref{td1}.

 We remark that Falconer and Grimmett \cite{FG} (see Theroem \ref{td88} below) proved that in $\mathbb{R}^d$, for $d \geq 2$ the orthogonal projections of the fractal percolation set to any coordinate planes is as big as possible (roughly speaking). This was extended in \cite{RS} to all projections simultaneously but only in the plane. Now we extend the result of \cite{FG} also in higher dimensions for all projections simultaneously.

It is well known that if all the probabilities $p_{i,j}$ are greater than $1/M$, then conditioned on non-emptiness, the fractal percolation set $E$ has Hausdorff dimension greater than $1$ a.s. \cite{F,MW}, and if all of the probabilities $p_{i,j}$ are smaller than a critical probability $p_c$, then $E$ is totally disconnected \cite{CD}. However, in \cite{RS} the authors gave a rather complicated technical condition which simplifies to $p>1/M$ when all $p_{i,j}$ equal $p$, under which the orthogonal projections of $E$ (which is a random dust) in all directions contain some interval, conditioned on $E\neq \emptyset$.

For $d\geq 3$, $d>k\geq 1$ the projections of $d$-dimensional fractal percolation to $k$-dimensional linear subspaces are more complicated. The aim of this note is to verify that the method of \cite{RS} can be extended to higher dimension under analogous assumptions. For example, one of the results of  \cite{RS} asserts that whenever the
Mandelbrot percolation set $E\subset \mathbb{R}^2$ has Hausdorff dimension  greater than one
and the two-dimensional sun shines at  $E$ (radial projection), then
 there is an interval in the shadow.
 However, we live in $\mathbb{R}^3$, so it is natural to verify the corresponding theorems  in higher dimension. In particular, it follows from our result that if the tree-dimensional Mandelbrot percolation $E\subset\mathbb{R}^3$ has Hausdorff dimension greater than two then for almost all realizations, at every moment (where ever the sun is) we
can find a disk in the shadow of $E$.

Although our proof follows the line of the proofs in \cite{RS} but here we needed to handle additional technical difficulties which do not appear in the plane.

\subsection{Notation}

\begin{figure}
  \begin{center}
  \captionsetup{singlelinecheck=off}
    \includegraphics[width=0.5\textwidth]{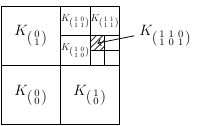}
    \caption{In dimension $d=2$, if $M=2$, then the indices are $2 \times n$ matrices with entries from $\{0,1\}$, for example the shaded level $3$ square is $K_{\left(\protect\begin{smallmatrix}1&1&0\\ 1&0&1\protect\end{smallmatrix}\right)}$.}\label{tdfig5}
    \end{center}
\end{figure}

We use the higher dimensional analogues of the notations of \cite{RS}. To define the fractal percolation in $[0,1]^d$,
first we label the $M^{-n}$-mesh cubes by $d\times n$ matrices chosen from
\[
\mathcal{A}_n=\{0,\dots,M-1\}^{d\times n},
\]
where the $1 \leq k \leq n$-th column corresponds to the level $k$ contribution. We explain this on a simple example with the help of Figure \ref{tdfig5}. Namely, in this example we assume that $d=2,M=2,n=3$ and $A=\left(\begin{smallmatrix}1&1&0\\ 1&0&1\end{smallmatrix}\right)$.
\begin{comment}
$A=\left(
           \begin{array}{ccc}
             1 & 1 & 0 \\
             1 &  0& 1 \\
           \end{array}
         \right)
$
\end{comment}
Then the first column $\left(\begin{smallmatrix}1\\1\end{smallmatrix}\right)$ of $A$ corresponds to the right top square of $K$ with side length $1/2$. The second column $\left(\begin{smallmatrix}1\\0\end{smallmatrix}\right)$ then corresponds to the right bottom smaller square of the previous square and the third column $\left(\begin{smallmatrix}0\\1\end{smallmatrix}\right)$ refers to the left top level $3$ square of its ancestor, which is the shaded square on Figure \ref{tdfig5}. Then the left bottom corner of this square is just
\[
\frac{1}{2} \cdot \left(
                    \begin{array}{c}
                      1 \\
                      1 \\
                    \end{array}
                  \right)
+
\frac{1}{2^2} \cdot \left(
                      \begin{array}{c}
                        1 \\
                        0 \\
                      \end{array}
                    \right)
+
\frac{1}{2^3} \cdot  \left(
                       \begin{array}{c}
                         0 \\
                         1 \\
                       \end{array}
                     \right).
\]
In general, for an $\mathbf{A}\in\mathcal{A}_n$ let $K_A$
be the corresponding level-$n$ cube. Then the homothety which maps the unit cube $K$ onto $K_A$ is
\[
  \varphi_\mathbf{A}(\mathbf{x})=\frac{1}{M^n} \mathbf{x}+ \mathbf{A}
\left(
\begin{array}{c}
M^{-1} \\
\vdots \\
M^{-n}
\end{array}
\right).
\]

We denote the $d$-dimensional Mandelbrot percolation in the unit cube $K$ with retain probabilities $\{p_\mathbf{A}\}_{\mathbf{A} \in \mathcal{A}_1}$ by $E=E(\omega)$. That is for $n\geq 0$ integers let $\mathcal{E}_n\subseteq \mathcal{A}_n$ be the random set defined inductively in the following way. Put $\mathcal{E}_0=\emptyset$. If for $\mathbf{A}\in \mathcal{A}_n$ we have $\mathbf{A}\notin \mathcal{E}_n$ then for any $\mathbf{C}\in \mathcal{A}_1$ and for $\mathbf{B}=
\big(
\begin{array}{c|c}
\mathbf{A} & \mathbf{C}
\end{array}
\big)
\in \mathcal{A}_{n+1}$ we have $\mathbf{B}\notin \mathcal{E}_{n+1}$ as well. On the other hand, if $\mathbf{A}\in \mathcal{E}_n$ then $\mathbf{B}=
\big(
\begin{array}{c|c}
\mathbf{A} & \mathbf{C}
\end{array}
\big)
\in \mathcal{A}_{n+1}$ with probability $p_\mathbf{C}$. The $n$-th approximation $E_n$ of $E$ is the subset of $K$ corresponding to $\mathcal{E}_n$:
\[
  E_n=\bigcup_{\mathbf{A} \in \mathcal{E}_n}{K_\mathbf{A}}.
\]
Then $E$ is defined by
\[
  E=\bigcap_{n=1}^\infty E_n.
\]

Now we turn our attention to the projections. Fix $1\leq k\leq d-1$. Let $\mathbf{a}^{(1)},\dots,\mathbf{a}^{(k)}$ be an orthonormal set of vectors in $\mathbb{R}^d$, put $\alpha=\left\{\mathbf{a}^{(1)},\dots,\mathbf{a}^{(k)}\right\}$  and let $P_{\alpha}$ be the linear subspace spanned by the vectors in $\alpha$:
\be \label{td87}
  S_\alpha=span\{\mathbf{a}^{(1)},\dots,\mathbf{a}^{(k)}\}.
\ee
Let $\gamma_\alpha=\left\{\bf{c}^{(1)},\dots,\bf{c}^{(d-k)}\right\}$ be
an arbitrary orthonormal basis of $P^\bot_{\alpha}$.

We consider the orthogonal projections $proj_{\alpha}$ of $E$ to each $k$-dimensional planes $S_{\alpha}$. Our goal is to determine the set of parameters $\{p_\mathbf{A}\}_{\mathbf{A} \in \mathcal{A}_1}$ for which almost surely $int\{ proj_{\alpha}E\} \neq \emptyset$ for all $\alpha$, conditioned on $E\neq \emptyset$.

It will be useful to handle projections parallel to some sides of the unite cube separately from other directions. Write $\mathbf{e}^{(i)}$ for the vector with all $0$ entries except for the $i$-th, which is $1$. We call $S_\alpha$ a \emph{coordinate plane} if there exist distinct $j_1,\dots,j_k \in[d]$ such that $S_\alpha=span\{\mathbf{e}_{j_1},\dots,\mathbf{e}_{j_k}\}$. As we mentioned above, the case of orthogonal projections to cordinate planes is fully covered by the paper of Falconer and Grimmett \cite[Theorem 1., p.3.]{FG}:

\begin{thm}[Falconer and Grimmett\cite{FG}]\label{td88}
Suppose that for all distinct $j_1,\dots,j_k\in[d]$ and for all $i^{(j_1)},\dots,i^{(j_k)}\in \{ 0,\dots,M-1\}$
\[
  \sum_{\substack{\mathbf{A} \in \mathcal{A}_1:\ \forall \ell \in [k]\\ A_{j_\ell,1}=i^{(j_\ell)}}}{p_{\mathbf{A}}}>1,
\]
Then almost surely, conditioned on $E\neq \emptyset$, for all distinct $j_1,\dots,j_k\in [d]$ the sets $proj_{(\mathbf{e}_{j_1},\dots,\mathbf{e}_{j_k})}(E)$ have nonempty interior.

On the other hand, if for distinct $j_1,\dots,j_k\in [d]$ and for $i^{(j_1)},\dots,i^{(j_k)}\in \{ 0,\dots,M-1\}$ we have
\[
  \sum_{\substack{\mathbf{A} \in \mathcal{A}_1:\ \forall \ell \in [k]\\ A_{j_\ell,1}=i^{(j_\ell)}}}{p_{\mathbf{A}}}<1,
\]
then almost surely the interior of $proj_{(\mathbf{e}_{j_1},\dots,\mathbf{e}_{j_k})}(E)$ is empty.
\end{thm}

In addition, we consider radial and co-radial projections as well.
\begin{defi}[Radial and co-radial projection]\label{td2345}
Given $t\in \mathbb{R}^d$:
\begin{itemize}
	\item the \emph{radial projection} with center $t$ of set $E$ is denoted by $\mathrm{Proj}_t(E)$ and is defined as the set of unit vectors under which points of $E\setminus\{t\}$ are visible from $t$. In particular, $\mathrm{Proj}_t$ maps to the $d-1$-sphere $S^{d-1}$.
	\item the \emph{co-radial projection} with center $t$ of set $E$ is denoted by $\mathrm{CProj}_t(E)$ and is defined as the set of distances between $t$ and points from $E$. In particular, $\mathrm{CProj}_t$ maps to $\mathbb{R}_+$.
\end{itemize}
\end{defi}

\subsection{Results}

 In Section \ref{td200} we  define Condition $A(\alpha)$ on the set of probabilities $\{p_\mathbf{A}\}_{\mathbf{A} \in \mathcal{A}_1}$. This is which is the key ingredient  for the following theorems. At this point it is enough to know that if $p_\mathbf{A}>M^{-(d-k)}$ for all $\mathbf{A} \in \mathcal{A}_1$, then Condition $A(\alpha)$ holds for all $\alpha$.
\begin{thm}\label{td19}
Fix $d\geq 2$ and $1\leq k<d$ and suppose that Condition $A(\alpha)$ holds for all $\alpha$ such that $S_\alpha$ is not a coordinate plane. In addition, to control parallel directions we suppose that for all distinct $j_1,\dots,j_k\in[d]$ and for all $i^{(j_1)},\dots,i^{(j_k)}\in \{ 0,\dots,M-1\}$
\[
  \sum_{\substack{\mathbf{A} \in \mathcal{A}_1:\ \forall \ell \in [k]\\ A_{j_\ell,1}=i^{(j_\ell)}}}{p_{\mathbf{A}}}>1,
\]
Then almost surely for \emph{all} $\alpha$ orthogonal projections $proj_\alpha(E)$ have nonempty interior, conditioned on $E\neq \emptyset$.
\end{thm}

\begin{thm}\label{td20}
Fix $d \geq 2$.
\begin{enumerate}
	\item Fix $k=d-1$ and suppose that Condition $A(\alpha)$ holds for all $\alpha$ such that $S_\alpha$ is not a coordinate plane. Then for almost all realizations of $E$ such that $E\neq \emptyset$, simultaneously for \emph{all} $t$ radial projections $\mathrm{Proj}_t(E)$ have nonempty interior.
	
	\item Fix $k=1$ and suppose that Condition $A(\alpha)$ holds for all $\alpha$ such that $S_\alpha$ is not a coordinate plane. Then for almost all realizations of $E$ such that $E\neq \emptyset$, simultaneously for \emph{all} $t$ co-radial projections $\mathrm{CProj}_t(E)$ have nonempty interior.
\end{enumerate}
\end{thm}

The rest of the paper is organized as follows. In Section \ref{td40} we consider orthogonal projections of the fractal percolation and prove Theorem \ref{td19}. Then in Section \ref{td41} we turn our attention to radial and co-radial projections, and using the same argument as in \cite{RS} we show that Theorem \ref{td20} holds.

\section{Orthogonal projections}\label{td40}

Since the range of the projection $proj_\alpha$ is different for different $\alpha$
it is more convenient  to substitute $proj_\alpha$ with a projection $\Pi_\alpha$
to some coordinate plane.

\subsection{Projection to coordinate axes}

Put $\binom{[d]}{k}$ for the $k$ element subsets of $[d]$. Then, for $\mathcal{I}=\{i_1,\dots,i_k\}\in \binom{[d]}{k}$ let $S_\mathcal{I}$ stand for the coordinate-plane spanned by the unit vectors corresponding to $\mathcal{I}$:
\be \label{td47}
  S_\mathcal{I}:=span\Big\{\mathbf e^{(i_1)},\dots,\mathbf e^{(i_k)} \Big\}.
\ee
Then $\Pi_{\alpha,\mathcal{I}}$ is the linear projection to $S_\mathcal{I}$ in direction $\gamma_\alpha$, that is for $\mathbf x \in \mathbb{R}^d$
\be \label{td48}
  \Pi_{\alpha,\mathcal{I}}(\mathbf x)= \Big\{span\{ \mathbf{c}^{(1)},\dots,\mathbf{c}^{(d-k)}\}+\mathbf x\Big\}\cap S_\mathcal{I},
\ee
where $\mathbf{c}^{(1)},\dots,\mathbf{c}^{(n-k)}$ are the vectors of $\gamma_\alpha$. We will see soon that if we choose a suitable $\mathcal{I}$, then the right-hand side of \eqref{td48} is a  single point.

To describe the projection in \eqref{td48} by matrix operations we introduce $\matr C=\matr{C}^\alpha$ as the $d\times (d-k)$ matrix whose column vectors are $\mathbf{c}^{1}, \dots ,\mathbf{c}^{d-k}$:
\[
\matr C=
\left(
\begin{array}{c|c|c}
\mathbf c^{(1)} & \dots & \mathbf c^{(d-k)}
\end{array}
\right).
\]
Let $\mathcal{I}^c:=[d]\setminus \mathcal{I}$ and let $i_1<\dots <i_k$ and $i_{k+1}<\dots <i_d$  be the elements of $\mathcal{I}$ and $\mathcal{I}^c$ respectively. We define the $d\times d$ matrix $\matr{M}(\mathcal{I})=\matr{M}^\alpha(\mathcal{I})$ by
\be \label{td56}
  \matr{M}(\mathcal{I}):=
\left(
\begin{array}{c|c}
\matr I & -\matr {C}_1 \\
\hline
\matr 0 & -\matr {C}_2
\end{array}
\right),
\ee
where $\matr I$ is the $k\times k$ identity matrix; $\matr 0$ is a $(d-k)\times k$ matrix with all zero entries; $\matr {C}_1=\matr{C}_1^\alpha$ is a $k\times (d-k)$ matrix whose $\ell$-th row is the $i_\ell$-th row of $\matr C$; and $\matr {C}_2=\matr{C}_2^\alpha$ is a $(d-k)\times (d-k)$ matrix whose $\ell$-th row is the $i_{k+\ell}$-th row of $\matr C$.

Using the Cauchy-Binet formula \cite[Section 4.6, page 208-214]{CB} we obtain that
\[
  1=\det\Big(\matr{C}^T\matr{C}\Big)=\sum_{\mathcal{I}\in \binom{[d]}{k}}{\bigg(\det\Big(\matr {C}_2(\mathcal{I})\Big)\bigg)^2}.
\]
Therefore there exists $\mathcal{I}'\in \binom{[d]}{k}$ such that $\det (\matr{M}(\mathcal{I'}))=-\det (\matr {C}_2(\mathcal{I'}))\neq 0$. We use the notation $\matr{M}=\matr{M}^\alpha:=\matr{M}^\alpha(\mathcal{I}')$. In order to find a formula for $\Pi_{\alpha,\mathcal{I}'}(\mathbf x)$ fix an $\mathbf x \in \mathbb{R}^d$ and let
\[
  \mathbf h:=(\matr{M}^\alpha)^{-1}\mathbf x.
\]
Then it is easy to see that
\be \label{td51}
  \Pi_{\alpha,\mathcal{I}'}(\mathbf x)=\sum_{j=1}^k{h_j\mathbf e^{(i_j)}}=\sum_{j=k+1}^d{h_{j}\mathbf c^{(j-k)}+\mathbf x}.
\ee

By symmetry, without any loss of generality we may assume that
\[
  \mathcal{I}'=\{1,\dots, k\}
\]
and restrict the set of directions $\alpha$ to
\be \label{td49}
  A_{\mathcal{I}'}=\Bigg\{\alpha\ \Bigg|\ \left| det\Big(\matr {C_2}(\mathcal{I}')\Big)\right| > \frac{1}{2\sqrt{\binom{d}{k}}}\Bigg\}.
\ee
We remark that in \eqref{td49} the matrix $C_2(\mathcal{I}')$ depends on $\alpha$.
Let $\Pi_\alpha$ be the projection to $S_{\mathcal{I}'}$ for $\alpha\in A_{\mathcal{I}'}$. For later computations we write \eqref{td51} in a more tractable form. For $\mathbf y\in \mathbb{R}^d$ let $\mathbf{y}_1\in \mathbb{R}^k$ and $\mathbf{y}_2\in \mathbb{R}^{d-k}$ be the vectors formed by the first $k$ and last $d-k$ elements of $\mathbf y$ respectively. Then for $\mathbf x\in \mathbb{R}^d$
\be \label{td89}
  \Pi_{\alpha}(\mathbf x)=\mathbf{x}_1-\matr{C}_1\ \matr{C}_2^{-1}\mathbf{x}_2.
\ee

 It is clear that for any $\alpha$,  $int\{ proj_{\alpha}E\} \neq \emptyset$ iff $int\{ \Pi_{\alpha}E\} \neq \emptyset$. In addition, $\Pi_{\alpha}E$ lays in the same plane for all $\alpha$, which will be useful when considering several directions at once, e.g. when considering nonlinear projections. Now we introduce the higher dimensional analogue of the notation used in \cite{RS}.

\subsection{Conditions A and B}\label{td200}
Let us denote by $\Delta_{\alpha}$ the $\Pi_\alpha$ projection of the unit cube $K$. For $\mathbf{A} \in \mathcal{A}_n$ we introduce the function $\psi_{\alpha,\mathbf{A}}:\Delta_\alpha\to \Delta_\alpha$ as the inverse of $\Pi_\alpha\circ \varphi_{\mathbf{A}}$.

The following operators are defined on functions from $\Delta_\alpha$ to nonnegative reals, vanishing on the boundary of $\Delta_\alpha$. These are one of the main tools of this paper, and are defined by
\[
  G_{\alpha,n} f(x)=\sum_{\mathbf{A} \in \mathcal{E}_n: x\in \Pi_\alpha (K_{\mathbf{A}})}
  {f\circ \psi_{\alpha,\mathbf{A}} (x)}.
\]
and given $G_{\alpha,n}$ we define $F_{\alpha,n}$ as
\[
  F_{\alpha,n}=\Ev\big( G_{\alpha,n}\big).
\]
That is, for $n=1$
\[
  F_\alpha f(x):=F_{\alpha,1}f(x)=\sum_{\mathbf{A} \in \mathcal{A}_1: x\in \Pi_\alpha (K_{\mathbf{A}})}{p_{\mathbf{A}}\cdot f\circ \psi_{\alpha,\mathbf{A}} (x)}.
\]
It is easy to see that $F_{n,\alpha}$ equals to the $n$-th iterate of $F_\alpha$:
\[
  F_{\alpha,n} f(x)=F_\alpha^n f(x)=\sum_{\mathbf{A} \in \mathcal{A}_n: x\in \Pi_\alpha (K_{\mathbf{A}})}
  {p_{\mathbf{A}}\cdot f\circ \psi_{\alpha,\mathbf{A}} (x)},
\]
where for $\mathbf{A}=
\left(
\begin{array}{c|c|c}
\mathbf a^{(1)} & \dots & \mathbf a^{(n)}
\end{array}
\right)
\in \mathcal{A}_n$ we write $p_{\mathbf{A}}=\prod_{j=1}^n{p_{\mathbf a^{(j)}}}$.

Now we present the higher dimensional analogue of
Conditions A and B of \cite{RS}.

\begin{defi}[Condition A]
We say that Condition $A(\alpha)$ holds if there exist $I_1^{\alpha},I_2^{\alpha}\subset \Delta_{\alpha}$ homothetic copyes of $\Delta_{\alpha}$
 with homoteties having the same centre as $\Delta_\alpha$,
  and there exists a positive integer $r=r_\alpha$ such that
\[ \ba
  &(i) I_1^{\alpha}\subset int\{ I_2^{\alpha}\},\ I_2^{\alpha}\subset int\{ \Delta_{\alpha}\},     \\
  &(ii) F^{r}_{\alpha}\ind_{I_1^{\alpha}}\geq 2\cdot\ind_{I_2^{\alpha}}.
\ea\]

\end{defi}
This is the place where the geometrical complexity of the problem differs from that of the original case in \cite{RS}: If $d=2$ and $k=1$, then $\Delta_{\alpha}$ is simply a line segment. However, if, for example $d=3$ and $k=2$, then $\Delta_{\alpha}$ is a hexagon, which carries some extra technical difficulties in the proof in the next section. Figure \ref{tdfig1} shows the mutual position of $\Delta_\alpha$, $I_2^{\alpha}$ and $I_1^{\alpha}$ for a fixed $\alpha$ and for $d=3$, $k=2$.

The following condition is stronger than Condition $A$, but it is easier to check.

\begin{defi}[Condition $B$]
We say that Condition $B(\alpha)$ holds if there exists a nonnegative continuous function $f: \Delta_{\alpha}\to \mathbb{R}$ such that $f$ vanishes exactly on the boundaries of $\Delta_{\alpha}$ and $\exists \varepsilon>0$:
\be \label{td2}
  F_{\alpha}f\geq (1+\varepsilon)f.
\ee
\end{defi}
In the following sections we show that Condition $B$ implies Condition $A$ (see Section \ref{td42}), which implies that $int\{ \Pi_{\alpha}E\} \neq \emptyset$ conditioned on $E\neq \emptyset$ (see Section \ref{td43}), and for certain choice of the parameters $\{p_\mathbf{A}\}_{\mathbf{A} \in \mathcal{A}_1}$ we show some functions $f$ satisfying Condition $B(\alpha)$ for all $\alpha$ (see Section \ref{td44}).

\subsection{Condition B implies Condition A}\label{td42}

\begin{prop}\label{td9}
Condition $B(\alpha)$ implies condition $A(\alpha)$ whenever
 $S_\alpha$ is not a coordinate plane.
\end{prop}

In order to verify Proposition \ref{td9}, first we state Lemma \ref{td1}. Then we prove Proposition \ref{td9} using Lemma \ref{td1}. Finally we prove Lemma \ref{td1}.

\begin{lemma}\label{td1}
  Suppose that condition $B(\alpha)$ holds for some $\alpha$ such that $S_\alpha$ is not a coordinate plane. Then there exists an integer $n>0$ and there exist sets $I_1^{\alpha},I_2^{\alpha}\subset \Delta_{\alpha}$ with the same properties as in the definition of Condition $A$ except  $(ii)$  is replaced with $(ii^*)$:
\[
  (ii^*) \forall x\in I_2^{\alpha}\ \ F^n_{\alpha}g_1(x)\geq \left( 1+\varepsilon \right)g_2(x),
\]
where $g_1=f|_{I_1}$, $g_2=f|_{I_2}$.
\end{lemma}

\begin{proof}[Proof of Proposition \ref{td9}]
Using Lemma \ref{td1}, Condition $A$ holds with the smallest multiple $r$ of $n$ satisfying
\[
  \left( 1+\varepsilon \right)^{r/n}\geq 2\frac{\max_{x\in I_1}{g_1(x)}}{\min_{x\in I_2}{g_2(x)}}.
\]
\end{proof}

\begin{comment}
  \begin{figure}[ht!]\label{tdfig1}
    \begin{center}
      \subfigure[$\Delta_\alpha$(continuous), $I_2^{\alpha}$(dashed) and $I_1^{\alpha}$(dotted).]{\includegraphics[width=0.35\textwidth]{Delta_2.pdf}}\label{tdfig3}
    \hspace{2cm}
      \subfigure[$W_n$ with $n=1$, $M=3$. The big bold faced green contour is $W_0$. The blue and the red small contours correspond to cubes $K_{[3,1,1]^T}$ and $K_{[3,3,2]^T}$, respectively. As their coordinates show, these cubes are separated from each other, however after projection their top and bottom faces intersect. Hence we have at least two \emph{rational} classes (see the definition in the proof of Lemma \ref{td1}): the sides on the top and on the bottom.]{\includegraphics[width=0.35\textwidth]{W_n_zold2.pdf}}
    \caption{We project from $d=3$ to $k=2$ dimension.}
    \end{center}
  \end{figure}
\end{comment}

  \begin{figure}[ht!]
    \begin{center}
      \subfigure[]{\includegraphics[width=0.35\textwidth]{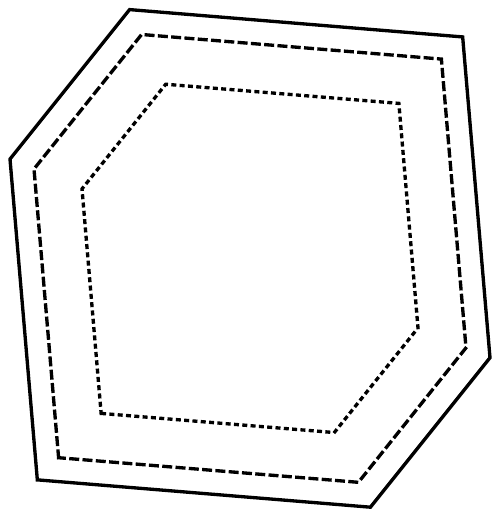}\label{tdfig1}}
    \hspace{2cm}
      \subfigure[]{\includegraphics[width=0.35\textwidth]{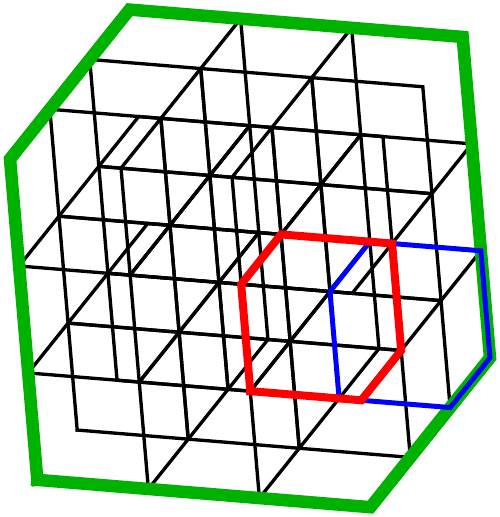}\label{tdfig3}}
    \caption{We project from $d=3$ to $k=2$ dimension. \textbf{(a)} $\Delta_\alpha$(continuous), $I_2^{\alpha}$(dashed) and $I_1^{\alpha}$(dotted). \textbf{(b)} $W_n$ with $n=1$, $M=3$. The big bold faced green contour is $W_0$. The blue and the red small contours correspond to cubes $K_{[3,1,1]^T}$ and $K_{[3,3,2]^T}$, respectively. As their coordinates show, these cubes are separated from each other, however after projection their top and bottom faces intersect. Hence we have at least two \emph{rational} classes (see the definition in the proof of Lemma \ref{td1}): the sides on the top and on the bottom.}
    \end{center}
  \end{figure}

To prove Lemma \ref{td1} we need the following definitions. Recall that we defined $\Delta_\alpha=\Pi_\alpha(K)$, where $K=[0,1]^d$. Since by assumption $S_\alpha$ is not a coordinate plane, the $\Pi_\alpha$ image of the $k-1$-dimensional faces
of the boundary of $K$ are $k-1$-dimensional. Then the boundary $W_0$ of $\Delta_\alpha$ is the union of $k-1$-dimensional faces.
The set of these $k-1$-dimensional faces (whose union form $W_0$) is denoted by $\mathcal{F}_0$.
In the special case shown on
Figure \ref{tdfig3} the boldfaced (green) hexagon is $W_0$  and the collection of the six sides is $\mathcal{F}_0$.
Let $K'$ be a level-$n$ cube. We define $W_{K'}$ and $\mathcal{F}_{K'}$ analogously. Let $W_n$ and $\mathcal{F}_n$ be the union of $W_{K'}$
and $\mathcal{F}_{K'}$ for all level $n$ cubes $K'$.

\emph{Sketch of the proof of Lemma \ref{td1}}: Basically we would like to follow the idea of the proof of \cite[Lemma 8]{RS}, but since the projection onto the $k$-dimensional plane is geometrically more complicated, now we explain how the proof is carried out. The main difference is that while the proof of case $d=2$, $k=1$ uses the fact that the sets $W_0$ and $W_n\setminus W_0$ are separated, the same is not true if $k>1$. Namely, the boldfaced (green) line on Figure \ref{tdfig3} is not separated from the union of the black (not boldfaced) lines.
This difficulty is handled by dividing  the sides of $W_0$ to \emph{rational} and \emph{irrational classes}.
For some $\alpha$ it is possible that there exist both rational and irrational sides. For rational classes some kind of periodicity occurs, while for irrational classes there is a separation similar to that of case $k=1$ and therefore we can use the continuity of function $f$.

\begin{proof}[Proof of Lemma \ref{td1}]
We order the faces of $\bigcup_{n=0}^\infty \mathcal{F}_n$ into equivalence classes in the following way.
 Every equivalence class can be identified with a face of $\mathcal{F}_0$. That is every face  $f_0\in\mathcal{F}_0$
determines an equivalence class.
 Let $f'\in \mathcal{F}_n$ for an $n \geq 1$. Then there exists a unique  level $n$ cube  $K'$ such that $f'$ is a $k-1$-dimensional face of $\Pi_\alpha (K')$.
 Let $f_0\in\mathcal{F}_0$ be the corresponding  $k-1$-dimensional face of $\Pi_\alpha (K)$. That is the relative position of $f'$ to $\Pi_\alpha(K')$ is the same as the relative position of $f_0$ to $\Pi_\alpha(K)$. Then $f'$ is equivalent to $f_0$.

Note that since most of the $k-1$-dimensional faces may belong to many different cubes, hence some different faces of the same level geometrically coincide, but it not causes any inconvenience.

Throughout this proof we use more tractable indices to denote the level $n$ cubes. For $n\geq 1$ let
\[
  \mathcal{B}_n=\{0,\dots,M^n-1\}^d.
\]
For $\mathbf i=(i_1,\dots,i_d) \in \mathcal{B}_n$ let
\[
  K_{\mathbf i}=M^{-n}[i_1,i_1+1]\times\dots\times[i_d,i_d+1].
\]

In this proof two types of classes are handled separately which are called \emph{rational} and \emph{irrational}. To define these consider the level $n$ cube $K_{\mathbf i}$. Let $f$ be one of the $k-1$-dimensional faces of $\Pi_\alpha(K_{\mathbf i})$. Its equivalence class is denoted by $\hat{f}_0$, where $f_0\in \mathcal{F}_0$. We write $Q_{\mathbf i}^{\hat{f}_0}$ for the $k-1$ dimensional plane which is spanned by $f$ (which is uniquely determined by $\mathbf i$ and $\hat{f}_0$). We say that the class $\hat{f}_0$ is \emph{rational} if there exists $n\geq1$ and $\mathbf{i},\mathbf{j}\in \mathcal{B}_n$ such that
\[
  Q^{\hat{f}_0}_{\mathbf i}=Q^{\hat{f}_0}_{\mathbf j} \text{\ \ , but\ \ \ } \mathbf i \neq \mathbf j.
\]
The other classes are called \emph{irrational}. An example of rational classes is shown on Figure \ref{tdfig3} in the case $d=3$, $k=2$ and $M=3$: for $n=1$ there exist two different cubes, the red and the blue ones, such that after projection the line of their top (and bottom) faces coincide.

First we discuss rational classes. We write $B_\rho(\mathbf x)$ for the ball centered at $\mathbf{x}\in\mathbb{R}^d$ with radius $\rho$.
We show that there exists $\vartheta_1,\eta_1> 0$ such that for all $n\geq 1$, for all $\eta_1>\eta>0$ and for all $\mathbf x\in \Delta_\alpha\setminus B_{\eta/M}(W_0)$
\be \label{td5} \ba
  \# \Big \{\mathbf i \in \mathcal{B}_n \ \Big | \ & \mathbf x\in \Pi_\alpha\left(K_{\mathbf i}\right)\ \&\ dist\left( \mathbf x\ ;\ Q^{\hat{f}_0}_{\mathbf i}\right)>\eta/M^n \text{ for all rational classes $\hat{f}_0$} \Big \}\geq \\
  &\geq \vartheta_1\cdot \# \Big \{\mathbf i \in \mathcal{B}_n \ \Big | \ \mathbf x\in \Pi_\alpha\left(K_{\mathbf i}\right) \Big \},
\ea \ee
where $dist(\ .\ ;\ .\ )$ is the usual distance in $\mathbb{R}^d$.
Now we verify the following Fact which asserts a kind of translation invariance.
\begin{fact}
  There exist $\hat{\mathbf j}\in \mathbb{Z}^d$ such that for all rational class $\hat{f}_0$, for all level $n$, for all ${\mathbf i}\in\mathcal{B}_n$ and for all $k\in \mathbb{Z}$ satisfying $\mathbf i+k\ \hat{\mathbf j} \in \mathcal{B}_n$ we have
\be \label{td13}
  dist\left( \mathbf x\ ;\ Q^{\hat{f}_0}_{\mathbf i}\right)=dist\left( \mathbf x\ ;\ Q^{\hat{f}_0}_{\mathbf i+k\:\hat{\mathbf j}}\right).
\ee
\end{fact}
\begin{proof}[Proof of the Fact]
Fix a rational class $\hat{f}_0$. Suppose that $Q^{\hat{f}_0}_{\hat{\mathbf i}}=Q^{\hat{f}_0}_{\widetilde{\mathbf i}}$ for some $\hat{\mathbf i}, \widetilde{\mathbf i} \in \mathcal{B}_m$, $m\geq 1$. Then for any $n\geq 1$ and for any $\mathbf i \in \mathcal{B}_n$
\[ Q^{\hat{f}_0}_{\mathbf i}=Q^{\hat{f}_0}_{\mathbf i+k(\widetilde{\mathbf i}-\hat{\mathbf i})}
\]
for any $k\in\mathbb{Z}$ as long as $\mathbf i+k(\widetilde{\mathbf i}-\hat{\mathbf i})\in \mathcal{B}_n$. Hence for all $\mathbf x\in \Delta_\alpha$ and for any indices $\mathbf i$
%\[
%  x\in \Pi_\alpha\left(K_{\hat{j}^{(1)},\dots,\hat{j}^{(d)}}\right),
%\]
%and
%\[
%  x\in \Pi_\alpha\left(K_{\hat{j}^{(1)}+k(\widetilde{i}^{(1)}-\hat{i}^{(1)}),\dots,\hat{j}^{(d)}+k(\widetilde{i}^{(d)}-\hat{i}^{(d)})}\right),
%\]
%then
\[
  dist\left( \mathbf x\ ;\ Q^{\hat{f}_0}_{\mathbf i}\right)=dist\left( \mathbf x\ ;\ Q^{\hat{f}_0}_{\mathbf i+k(\widetilde{\mathbf i}-\hat{\mathbf i})}\right)
\]
holds since the planes are the same. It is easy to see that the same is true while considering all rational classes at once.
\end{proof}

Now we continue  the proof of Lemma \ref{td1}.

  \begin{figure}[ht!]
    \begin{center}
      \subfigure[]{\includegraphics[width=0.35\textwidth]{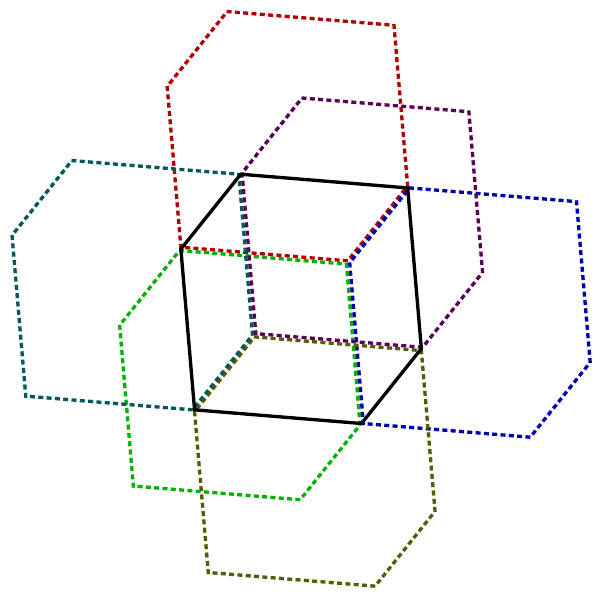}\label{tdfig2}}
    \hspace{2cm}
      \subfigure[]{\includegraphics[width=0.35\textwidth]{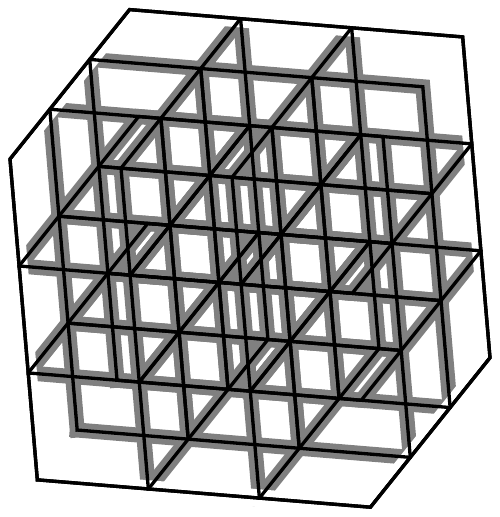}\label{tdfig4}}
    \caption{We project from $d=3$ to $k=2$ dimension. \textbf{(a)} The covering of $\Pi_\alpha(K_0)$(continuous) with the open sets $V_i$(dotted), $i=0,\dots,6$. \textbf{(b)} The gray contour of the inner border is $I_2\cap B_{\eta/M^{n'}}(W_{n'})$.}
    \end{center}
  \end{figure}

\begin{comment}
  \begin{figure}
    \includegraphics[width=0.5\textwidth]{Cover_szines_3.pdf}
    \caption{We project from $d=3$ to $k=2$ dimension. The covering of $\Pi_\alpha(K_0)$(continuous) with the open sets $V_i$(dotted), $i=0,\dots,6$.}\label{tdfig2}
  \end{figure}

  \begin{figure}
    \includegraphics[width=0.35\textwidth]{hatszog_abrak_14_06_17_inner_border.pdf}
    \caption{We project from $d=3$ to $k=2$ dimension. The gray contour of the inner border is $I_2\cap B_{\eta/M^{n'}}(W_{n'})$.}\label{tdfig4}
  \end{figure}
\end{comment}

 In order to show that the set in the first line of \eqref{td5} is nonempty we prove the following Fact:
\begin{fact}
  There exist a constant $\eta_1>0$ such that the following holds:

  Let $\mathbf x\in \Delta_\alpha\setminus B_{\eta_1/M}(W_0)$ and $n\geq 1$.
 Let $K_{\mathbf{i}}$ be an arbitrary level-$n$ cube such that $\mathbf x\in \Pi_\alpha(K_\mathbf{i})$. Then at least one of the neighbors $K_{\mathbf{i}'}$ of $K_{\mathbf{i}}$ satisfies that
 $$
 B_{\eta_1}(\mathbf{x})\subset \Pi_\alpha(K_{\mathbf{i}'}),
 $$
 where the level-$n$ cubes $K_\mathbf{i}$ and $K_{\mathbf{i}}'$
 are neighbors if they share a common $d-1$-dimensional face.

\end{fact}

\begin{proof}[Proof of the Fact]
For notational simplicity we write $K_0:=K_\mathbf{i}$ and we denote the neighbors of $K_0$ by $K_1,\dots,K_{2d}$.
Further, we denote by $V_i$, $i=0,\dots,2d$ the open shadows of these cubes, i.e.
\[
  V_i=int\ \Pi_\alpha(K_i).
\]
Using that $S_\alpha$ is not a coordinate plane it is easy to see that the union of the sets $V_i$ cover $\Pi_\alpha(K_0)$. Figure \ref{tdfig2} shows $\Pi_\alpha(K_0)$ and the covering for $d=3$ and $k=2$.
This is so, because by the assumption b

 Hence we can choose $\eta_1$ to be smaller than $M^n$ times the Lebesgue number of the finite open covering $\{V_i\}_{i=0}^{2d}$, which is always positive \cite[Theorem 0.20]{W}, and by similarity proportional to $M^{-n}$. For the definition of the Lebesgue number of a covering, see \cite{W}. Note that if one of the cubes $K_i$, $i=1,\dots,2d$ is out of the unit cube, then we don't need that cube for the cover, since our goal is to cover $K_0\setminus B_{\eta_1/M}(W_0)$. Another note is that the same argument remains valid for any $\eta$ such that $\eta_1>\eta>0$ holds.
\end{proof}

\begin{comment}
The next observation is very simple. If $n_1=n_1(\eta_1)$ is big enough and $n>n_1$ then a whole period of a level $n$ cube which is not close to the border is in the unit cube: There exists $n_1=n_1(\eta_1)$ such that for all level $n>n_1$ for all cube $K_{\mathbf i}$ such that $dist(\Pi_\alpha(K_{\mathbf i}), W_0)\geq \eta_1/M$ there exists $k=\pm 1$ such that
\be \label{td14}
  0\leq i^{(l)}+k\ \hat{j}^{(l)}\leq M^{n}
\ee
for all $l=1,\dots,d$.
\end{comment}

Now we are able to handle rational classes. By the definition of $\eta_1$ for all $\mathbf x\in \Delta_\alpha\setminus B_{\eta_1/M}(W_0)$ and for all $n\geq 1$ there exists $\mathbf i \in \mathcal{B}_n$ such that
\[
  B_{\eta_1/M^n}(\mathbf x)\subseteq \Pi_\alpha(K_{\mathbf i}).
\]
In addition, using the periodicity described in \eqref{td13}, cubes with such property follow each other periodically with the period independent of $\mathbf x$ and independent of $n$. Hence we conclude that the assertion in \eqref{td5} holds: There exists $\vartheta>0$ such that for any $\eta_1>\eta>0$, for any level $n\geq 1$ and for all $\mathbf x\in \Delta_\alpha\setminus B_{\eta/M}(W_0)$
\[ \ba
  \# \Big \{\mathbf i \in \mathcal{B}_n \ \Big | \ & \mathbf x\in \Pi_\alpha\left(K_{\mathbf i}\right)\ \&\ dist\left( \mathbf x\ ;\ Q^{\hat{f}_0}_{\mathbf i}\right)>\eta/M^n \ \forall \text{ rational classes $\hat{f}_0$} \Big \}\geq \\
  &\geq \vartheta \# \Big \{\mathbf i \in \mathcal{B}_n \ \Big | \ \mathbf x\in \Pi_\alpha\left(K_{\mathbf i}\right) \Big \}.
\ea \]
Note that $\vartheta$ does not depend on $\eta_1$, only the length of a period has effect on it. This equation will be sufficient to handle rational classes.

To handle irrational classes as well let us denote by $C_2$ the number of these classes. By the assumption of the lemma, for $\alpha$ fixed Condition $B(\alpha)$ holds for some function $f$ and $\varepsilon>0$. It is clear that there exists an integer $n'\geq 1$ such that
\be \label{td4}
  \vartheta(1+\varepsilon)^{n'}-C_2 > 1+\varepsilon,
\ee
where $\varepsilon$ was defined in Condition $B(\alpha)$. Choose $\eta_2>0$ to be smaller than the half of the smallest distance between any two level $n'$ faces falling in the same irrational class and to have
\be \label{td11}
  \inf_{\mathbf x\in \Delta_\alpha\setminus B_{\eta_2/M}(W_0)}{f(\mathbf x)}\geq \sup_{\mathbf x\in B_{\eta_2/M^{n'}}(W_0)}{f(\mathbf x)},
\ee
which can be easily achieved by the continuity of $f$. Setting $\eta=\min\{\eta_1,\eta_2\}$, we define the sets wanted in Condition A by
\[
  I_1:=\Delta_{\alpha}\setminus B_{\eta}(W_0)\text{ and by }I_2:=\Delta_{\alpha}\setminus B_{\eta/M}(W_0).
\]

When we verify that the assertion of the lemma holds we distinguish two cases: first we show it for those $x\in I_2$ which are separated from the inner borders $W_{n'}$, and then for those which are "close" to them, see Figure \ref{tdfig4}. The first case is obvious. However,
 in the second case we have to handle the following difficulty:  When $\mathbf x$ is close to $W_0$, then $f(\mathbf x)$ is close to zero. However, for fixed $\mathbf x\in I_2$, both in case of rational and irrational faces, we have bounds on the number of level $n'$ small cubes such that $\mathbf x$ is close to a face of that cube. Namely,
\begin{itemize}
	\item In the first case when  $\mathbf x\in I_2\setminus B_{\eta/M^{n'}}(W_{n'})$, we  use the definition of $F_{\alpha}$ and \eqref{td2} to obtain
\[
  F^{n'}_{\alpha}g_1(\mathbf x)=F^{n'}_{\alpha}f(\mathbf x)\geq (1+\varepsilon)^{n'} f(\mathbf x)>\left( 1+\varepsilon\right)g_2(\mathbf x).
\]

  \item In the second case $\mathbf x\in I_2\cap B_{\eta/M^{n'}}(W_{n'})$. This is the place where we use all the former preparations. Putting together the definition of $\eta$ and \eqref{td5} we obtain   that from one irrational class at most one cube can cover $\mathbf x$ close to its boundary we get
\[ \ba
  \# \Big \{\mathbf i \in \mathcal{B}_{n'} \ \Big | \ & \mathbf x\in \Pi_\alpha\left(\varphi_{\mathbf i}\left(B_{\eta}(W_0)\right)\right) \Big \}\leq \\
  &\leq (1-\vartheta) \# \Big  \{ \mathbf i \in \mathcal{B}_{n'} \ \Big | \ \mathbf x\in \Pi_{\alpha}\left(K_{\mathbf i}\right) \Big \} +C_2.
\ea \]
Putting this together with \eqref{td11} and with the definition of $F_{\alpha}$ we obtain
\be \ba \label{td3}
  F^{n'}_{\alpha}g_1(\mathbf x)&\geq F^{n'}_{\alpha}f(\mathbf x)-(1-\vartheta)F^{n'}_{\alpha}f(\mathbf x)-C_2f(\mathbf x) \\
  &= \vartheta F^{n'}_{\alpha}f(\mathbf x)-C_2f(\mathbf x)\ \forall \mathbf x\in\Delta_{\alpha}.
\ea \ee
Then by \eqref{td3}, \eqref{td2} and \eqref{td4} we have
\[
  F^{n'}_{\alpha}g_1(\mathbf x)\geq \left(1+\varepsilon \right)f(\mathbf x)+\left(\vartheta( 1+\varepsilon)^{n'}-1-\varepsilon-C_2\right)f(\mathbf x)\geq \left( 1+\varepsilon \right)g_2(\mathbf x).
\]
\end{itemize}
\end{proof}

\subsection{Examples}\label{td44}

In this section we show three examples when Condition B($\alpha$) can be checked.
\begin{enumerate}
	\item [Ex. 1]
The case of equal probabilities can be handled as in \cite{RS}: Suppose that $p_{\mathbf A}=p>1/M^{d-k}$ for all $\mathbf A\in \mathcal{A}_1$. Then let us define the function $f:\Delta_\alpha\to \mathbb{R}^+$ that we can use in Condition $B(\alpha)$ by
\be \label{td99}
  f(\mathbf x)=\left|\psi^{-1}_{\alpha}(\mathbf x) \cap K\right|,
\ee
where $\left|.\right|$ denotes the $d-k$-dimensional Lebesgue measure. Clearly, $f$ vanishes continuously on the borders of $\Delta_\alpha$, and strictly positive inside. In addition, it is obvious that
\[
  F_{\alpha}f= M^{d-k}p\cdot f.
\]
Therefore the requirements of Condition $B(\alpha)$ are satisfied for all $\alpha$ and hence we can apply Theorem \ref{td20}. Moreover, the above example is sharp, because for $p< 1/M^{d-k}$ we have $dim_H E< k$.

\item [Ex. 2]

We give another example, when the function defined  in \eqref{td99} satisfies Condition $B(\alpha)$. We divide the $3$-dimensional unit cube into 27 congruent cubes of sides  $1/3$. Then we remove the small cube in the very center of the unit cube with some probability $p$, and the remaining 26 small cubes are retained with probability $q$. Finally, we project it to planes. Namely, let $d=3$, $M=3$ and $k=2$. Let
\[
  p_{(i,j,k)^T}=\begin{cases} p\text{, if $i=j=k=1$,}\\
                          q\text{ otherwise.}\end{cases}
\]

If we project it orthogonally to a coordinate-plane, then by Theorem \ref{td88} we know that if $p+2q<1$ or $3q<1$, then almost surely there is no interval in the projected fractal. On the other hand, if both $p+2q>1$ and $3q>1$ hold, then it is easy to see that the function $f$ in \eqref{td99} satisfies Condition B($\alpha$) for all $\alpha$, and hence almost surely we have an interval in the projected fractal in all directions, conditioned on $E\neq \emptyset$.

\end{enumerate}

\subsection{Condition A implies nonempty interior}\label{td43}

\subsubsection{Robustness}

To handle all directions at once we show that the robustness property described in \cite[Section 4.3]{RS} holds in the higher dimensional case as well. Suppose that condition $A(\alpha)$ holds for some $\alpha=\left\{\mathbf{a}^{(1)},\dots,\mathbf{a}^{(k)}\right\}$ with $I_1^\alpha$, $I_2^\alpha$ and $r=r_\alpha$. Let $\delta$ be the Hausdorff distance between $I_1^\alpha$ and $I_2^\alpha$ and let $I_1'$ be the $\delta/2$ neighborhood of $I_1^\alpha$. In what follows   we show that condition $A(\alpha)$ also holds in some neighborhood of $\alpha$. First we define an equivalence relation on the set of directions as follows: for an $\widetilde{\alpha}=\left\{\widetilde {\mathbf{a}}^{(1)},\dots,\widetilde{\mathbf{a}}^{(k)}\right\} \in \mathbb{R}^{d\times k}$ we have
\[
  \alpha\sim \widetilde{\alpha} \Leftrightarrow S_\alpha=S_{\widetilde{\alpha}},
\]
where the plane $S_\alpha$ was defined in \eqref{td87}. Then we define a distance between $\alpha$ and $\beta=\left\{\mathbf{b}^{(1)},\dots,\mathbf{b}^{(k)}\right\}$ by
\be \label{td57}
  d(\alpha,\beta)=\min_{\substack{\widetilde{\alpha}\sim \alpha\\ \widetilde{\beta}\sim \beta}}\max_{i=1,\dots,k}\left\| \widetilde {\mathbf{a}}^{(i)}-\widetilde {\mathbf{b}}^{(i)} \right\|_2.
\ee

\begin{prop}{Robustness.}\label{td23}
Fix $\alpha$, $I_1^\alpha$, $I_2^\alpha$, $\delta$ and $I_1'$ as before. Then there exists a constant $0<C_R<\infty$ such that for all $\varepsilon>0$ if $d(\alpha,\beta)<C_R \cdot\varepsilon$, then
\[
  \left\| \Pi_\alpha \mathbf{x}-\Pi_\beta \mathbf{x}\right\|_2 <  \varepsilon \quad\ \text{for all\ } \mathbf x\in K.
\]
\end{prop}
\begin{proof}
Recall that $\Pi_\alpha \mathbf{x}$ was defined by formulas in \eqref{td89}. Thus
\be \label{td50}
  \Big\| \Pi_\alpha \mathbf{x}-\Pi_\beta \mathbf{x}\Big\|_2 \leq \left\|\matr {C}_1^\alpha \big(\matr {C}_2^\alpha\big)^{-1}-\matr {C}_1^\beta \big(\matr {C}_2^\beta\big)^{-1}\right\|\left\|\mathbf x\right\|_2,
\ee
where $\left\| .\right\|$ denotes the induced norm of $\left\| .\right\|_2$. The key observation in the following computations is that the determinants of $\matr {C}_2^\alpha$ and $\matr {C}_2^\beta$ are separated from zero.

Clearly $\left\|\mathbf x\right\|_2\leq \sqrt{d}$. To give a bound on the norm of the matrix on the right-hand side of \eqref{td50} we divide it into two parts:
\be \ba \label{td59}
  \left\|\matr {C}_1^\alpha \big(\matr {C}_2^\alpha\big)^{-1}-\matr {C}_1^\beta\big(\matr {C}_2^\beta\big)^{-1}\right\|&\leq
   \left\|\matr {C}_1^\alpha \big(\matr {C}_2^\alpha\big)^{-1}-\matr {C}_1^\alpha\big(\matr {C}_2^\beta\big)^{-1}\right\|\\
   &+\left\|\matr {C}_1^\alpha \big(\matr {C}_2^\beta\big)^{-1}-\matr {C}_1^\beta\big(\matr {C}_2^\beta\big)^{-1}\right\|.
\ea \ee
Regarding the first part of the right-hand side of \eqref{td59} we have
\be \ba \label{td58}
\bigg\| \big(\matr {C}_2^\beta\big)^{-1}&-\big(\matr {C}_2^\alpha\big)^{-1} \bigg\|
=\left\| \frac{adj(\matr {C}_2^\beta)}{det(\matr {C}_2^\beta)}-\frac{adj(\matr {C}_2^\alpha)}{det(\matr {C}_2^\alpha)}\right\|\\
&\leq \left\| \frac{adj(\matr {C}_2^\beta)}{det(\matr {C}_2^\beta)}-\frac{adj(\matr {C}_2^\alpha)}{det(\matr {C}_2^\beta)}\right\|
+ \left\| \frac{adj(\matr {C}_2^\alpha)}{det(\matr {C}_2^\beta)}-\frac{adj(\matr {C}_2^\alpha)}{det(\matr {C}_2^\alpha)}\right\|\\
&= \frac{1}{\left|det{\matr {C}_2^\beta}\right|}\left\| adj(\matr {C}_2^\alpha)-adj(\matr {C}_2^\beta)\right\|+\frac{\left\| adj(\matr {C}_2^\alpha)\right\|}{\left|det(\matr {C}_2^\alpha)det(\matr {C}_2^\beta)\right|}\left|det(\matr {C}_2^\alpha)-det(\matr {C}_2^\beta)\right|.
\ea \ee
It is easy to see that if $d(\alpha,\beta)<\varepsilon$, then the elements of the matrices $\matr {C}_2^\alpha$ and $\matr {C}_2^\beta$ differ from each other by at most constant times $\varepsilon$. Therefore, there exist constants $c_1,c_2$ such that
\[
\left|det{\matr {C}_2^\alpha}-det{\matr {C}_2^\beta}\right|<c_1\varepsilon \text{ and } \left\| adj(\matr {C}_2^\alpha)-adj(\matr {C}_2^\beta)\right\|<c_2\varepsilon.
\]
In addition, by \eqref{td49} the determinants are uniformly separated from zero, and the matrices $\matr {C}_2^\alpha$ and $\matr {C}_2^\beta$ have entries in $[0,1]$, which enables us to give a uniform (for all $\alpha$) upper bound for $\left\| adj(\matr {C}_2^\alpha)\right\|$. Hence if $d(\alpha,\beta)<C_R\cdot \varepsilon$ with $C_R$ small enough, then we have
\be \label{td100}
  \left\| \big(\matr {C}_2^\beta\big)^{-1}-\big(\matr {C}_2^\alpha\big)^{-1} \right\|\leq \frac{\varepsilon}{2\sqrt{d}}.
\ee

Similarly the second part of the right-hand side of \eqref{td59} is small as well:
\be \label{td101}
  \left\|\matr {C}_1^\alpha(\matr {C}_2^\beta)^{-1}-\matr {C}_1^\beta(\matr {C}_2^\beta)^{-1}\right\|\leq \left\|(\matr {C}_2^\beta)^{-1}\right\| \left\|\matr {C}_1^\alpha-\matr {C}_1^\beta\right\|\leq \frac{\varepsilon}{2\sqrt{d}},
\ee
if $C_R$ is small enough. Putting together \eqref{td50}, \eqref{td59}, \eqref{td100} and \eqref{td101} concludes the proof.
\end{proof}

\begin{cor}\label{td55}
Setting $\varepsilon=M^{-nr}\delta/2$ yields that if
\[
  d(\alpha,\beta) < C_R M^{-nr}\delta/4,
\]
then
\[
  \left\| \Pi_\alpha \mathbf{x}-\Pi_\beta \mathbf{x}\right\|_2 <  M^{-nr} \delta/4 \quad\ \text{for all\ } \mathbf x\in K,
\]
and hence for any $\mathbf A\in \mathcal{A}_{nr}$
\[
  \Pi_\beta\circ \varphi_{\mathbf A}(I_1^\beta)\subset \Pi_\beta \circ \varphi_{\mathbf A}(I_1').
\]
In particular the $n=1$ case implies that Condition $A$ holds for all directions in
\[ \ba
  J=B_{C_R M^{-r}\delta/2}(\alpha)\cap A_{S'}
%  J=&\left(a^{(1)}-\delta M^{-r}/(c(k,d)\cdot k\cdot d^2),a^{(1)}+\delta M^{-r}/(c(k,d)\cdot k\cdot d^2)\right)\times \dots \\
%  & \times \left(a^{(d)}-\delta M^{-r}/(c(k,d)\cdot k\cdot d^2),a^{(d)}+\delta M^{-r}/(c(k,d)\cdot k\cdot d^2)\right)
\ea \]
with the same $I_1'$, $I_2^\alpha$ and $r$.
\end{cor}

\subsubsection{The main proof}\label{td17}

Hence we can restrict ourselves to a range $J$ like above. The proof follows the line of the proof in \cite[Section 5]{RS}. Note that it is enough to prove that the nonempty interior in Theorem \ref{td19} exists with positive probability. This is because almost surely conditioned on $E\neq \emptyset$, for any $N$ there exists $n$ such that there are at least $N$ retained level $n$ cubes which will not vanish totally. In addition, events happening in different cubes are independent and statistically similar. Hence if the interior of all orthogonal projections is nonempty with positive probability, then the same holds almost surely.

\begin{proof}[Proof of Theorem \ref{td19}]

Let us write $I_1$ for $I_1'$, $I_2$ for $I_2^\alpha$ and $\delta'$ for the new Hausdorff distance $\delta/2$. Assume $\mathbf x,\mathbf y\in \Delta_\alpha$, $\left\| \mathbf x-\mathbf y\right\|_2<\delta' M^{-nr}/4$, and that for $\alpha, \beta\in J$
\[
  d(\alpha,\beta) < C_RM^{-nr}\delta'/4.
\]
Then from Corollary \ref{td55} by triangle inequality it follows that $\left\|\Pi_\alpha(\mathbf x)-\Pi_\beta(\mathbf y)\right\|_2<\delta' M^{-nr}/2$, and hence for any $\mathbf A\in \mathcal{A}_{nr}$
\be\label {td10}
  G_\beta^r\ind_{I_2}(\psi_{\mathbf A}(x))\geq G_\alpha^r\ind_{I_1}(\psi_{\mathbf A}(y)).
\ee
For given $n$ let $X_n$ be a $\delta' M^{-nr}/4$ dense subset of $I_1$ and $Y_n$ be a $C_RM^{-nr}\delta'/4$ dense subset of $J$ such that
\[
  \#(X_n\times Y_n)\leq cM^{2kdnr},
\]
with some constant $c$. For any $(x,\theta)\in I_1'\times J$ we define a sequence of random variables
\[
  V_n(x,\theta)=\#\Big\{ \mathbf A\in \mathcal{E}_{nr}\ \big|\ x\in \Pi_\theta\circ \varphi_{\mathbf A}(I_2)\Big\},
\]
where $\mathcal{E}_n=\mathcal{E}_n(\omega)$ stands for the set of retained level $n$ cubes. We prove that with positive probability $V_n(\mathbf x,\theta)\geq(3/2)^n$ for all $n,\mathbf x,\theta$. Let us use induction on $n$ and note that the $n=0$ case is obvious. For $(\mathbf y,\kappa)\in X_{n+1}\times Y_{n+1}$ let
\[
  Z(y,\kappa)=\Big\{(\mathbf x,\theta)\in I_1'\times J\ \Big|\ \left\| \mathbf x-\mathbf y\right\|_2<\delta' M^{-(n+1)r}/4\ \&\ d(\theta-\kappa)< C_RM^{-(n+1)r}/4\Big\}.
\]
Clearly the sets $Z(y,\kappa)$ cover $I_1'\times J$.

By the inductive hypothesis with positive probability $V_n(\mathbf x,\theta)\geq (3/2)^n$. For each level $nr$ cube $K_{\mathbf A}$ in $V_n(\mathbf x,\theta)$ the number of its sub-cubes in $V_{n+1}(\mathbf x,\theta)$ is given by $G_\theta^r\ind_{I_2}(\psi_{\mathbf A}(\mathbf x))$, which by \eqref{td10} can be bounded from below uniformly in $(\mathbf x,\theta)\in Z(\mathbf y,\kappa)$ by
\[
G_\theta^r\ind_{I_2}(\psi_{\mathbf A}(\mathbf x))\geq G_\kappa^r\ind_{I_1}(\psi_{\mathbf A}(\mathbf y)).
\]
The expected value of this random variable is $2$, and it is bounded below by $0$, above by $M^{dr}$. Moreover, random variables coming from different level $nr$ cubes are independent. Hence, by Azuma-Hoeffding inequality
\[
  \Pv\left( \sum_{\substack{\mathbf A\in \mathcal{E}_{nr}:\\ x\in \Pi_\theta\circ \varphi_{\mathbf A}(I_2)}} {G_\kappa^r\ind_{I_1}(\psi_{\mathbf A}(\mathbf y))} \geq \left(\frac{3}{2}\right)^{n+1} \Bigg| \ V_n(\mathbf y,\kappa)\geq \left(\frac{3}{2}\right)^n \right)\geq 1-\rho^{\left( 3/2\right)^n},
\]
where $0<\rho<1$ is fixed. Hence using the notation
\[
  E_n=\Bigg\{(\forall \mathbf x\in X_{n})(\forall \theta\in Y_{n})\ V_{n}(\mathbf x,\theta)\geq \left(\frac{3}{2}\right)^{n}\Bigg\}
\]
we have
\[
  \Pv\Big( E_{n+1}\ \big|\ E_n\Big)\geq \left(1-\rho^{\left( 3/2\right)^n}\right)^{cM^{2kdnr}}.
\]
Summation in $n$ converges, hence for any $\{p_\mathbf{A}\}_{\mathbf{A} \in \mathcal{A}_1}$ satisfying Condition $B$, a.s. $\Pi_\alpha(E(\omega))$ has nonempty interior for all $\alpha$ such that $S_\alpha$ is not a coordinate plane, conditioned on $E\neq \emptyset$.

\end{proof}

\section{Radial and co-radial projections}\label{td41}

In this section we consider radial and co-radial projections $Proj_t(E)$ and $CProj_t(E)$ with center $t\in \mathbb{R}^d$ of $E$. Recall that $Proj_t(E)$ is the set of vectors under which points of $E\setminus \{ t\}$ are visible from $t$ and $CProj_t(E)$ is the set of distances between $t$ and points from $E$, see Definition \ref{td2345}. Our goal is to prove Theorem \ref{td20}. We do this as it is done in \cite{RS} for the $2$-dimensional case: We introduce a notion called \emph{Almost linear family of projections} for which, using the robustness property, it is easy to show that almost surely for all member of the family the interior of the projection of the percolation fractal is nonempty. Then we show that radial and co-radial projections can be viewed as \emph{Almost linear families of projections}.

\subsection{Almost linear family of projections}

We fix the dimensions $d$ and $k$ and recall that $A_{\{1,\dots,k\}}$ was defined in \eqref{td49}. Consider a parametrized family of projections
\be \label{td22}
  S_t(\mathbf x):K\to \bigcup_{\alpha\in A_{\{1,\dots,k\}}}{\Delta_\alpha}\subset span\{\mathbf e^{(1)},\dots,\mathbf e^{(k)}\},
\ee
$t\in T$. For all $\mathbf x\in K$ let $\alpha_t(\mathbf x)$ be such that
\be \label{td16}
  S_t(\mathbf x)=\Pi_{\alpha_t(\mathbf x)}(\mathbf x).
\ee
Note that $span\{\alpha_t(\mathbf x)\}=S_{\alpha_t(\mathbf x)}$ is not always well defined, since the only restriction on it is to have $(S_t(\mathbf x)-\mathbf x)\in P_{\alpha_t(\mathbf x)}$ and $\alpha_t(\mathbf x)\in A_{\{1,\dots,k\}}$. We suppose that $\alpha_t(\mathbf x)$ is such that $S_{\alpha_t(\mathbf x)}$ is not a coordinate plane.
\begin{defi}[Almost linear family of projections]\label{td18}
We say that a family $\{ S_t\}_{t\in T}$ ($S_t$ satisfies \eqref{td22}) is an \emph{almost linear family of projections} if we can choose $\alpha_t(\mathbf x)$ (according to \eqref{td16}) in such a way that the following properties are satisfied. We set $J$ as the range of vectors for which Condition $A(\alpha)$ is satisfied with the same $I_1,I_2$ and $r$. We denote by $\delta$ the Hausdorff distance between $I_1$ and $I_2$.
\begin{enumerate}
	\item[i)] $\alpha_t(\mathbf x)\in J$ for all $t\in T$ and $\mathbf x\in K$.
	\item[ii)] $\alpha_t(\mathbf x)$ is a Lipschitz function of $\mathbf x$, with the Lipschitz constant not greater than $C_R(J)\delta/4$. This guarantees in particular that $S_t(K_{\mathbf A})$ is connected for any $n$ and $\mathbf A\in \mathcal{A}_n$.
	\item[iii)] For any $n$ we can divide $T$ into subsets $Z_i^{(n)}$ such that whenever $t,s\in Z_i^{(n)}$ and $\mathbf x,\mathbf y\in K_{\mathbf A}$, $\mathbf A\in \mathcal{A}_n$, we have
	\[
	  \left\| \alpha_t(\mathbf x)-\alpha_s(\mathbf y)\right\| \leq C_R(J)M^{-nr}\delta/4.
	\]
	Moreover, we can do that in such a way that $\#\{ Z_i^{(n)}\}$ grows only exponentially fast with $n$.
\end{enumerate}
\end{defi}
In the following we show that:

\begin{thm}\label{td70} Suppose that Condition $A(\alpha)$ holds for all $\alpha \in J$. Then for an almost linear family of projections $\{ S_t\}_{t\in T}$ almost surely $S_t(E)$ has nonempty interior for all $t$ conditioned on $E\neq \emptyset$.
\end{thm}
The proof follows the proof of \cite[Theorem 14]{RS} and is a modified version of the proof in Section \ref{td17}.

\begin{proof}

Let
\[
  V_n(\mathbf x,t)=\#\Big\{ \mathbf A\in \mathcal{E}_{nr}\ \big|\ \mathbf x\in S_t\circ \varphi_{\mathbf A}(I_2)\Big\}.
\]
We prove inductively that with positive probability $V_n(\mathbf x,\theta)\geq(3/2)^n$ for all $n,\mathbf x,\theta$. The $n=0$ case is obvious, a.s. $V_0(\mathbf x,\theta)=1$ for all $\mathbf x,\theta$. For given $n$ let $X_n$ be a $\delta M^{-nr}/4$ dense subset of $I_1$. Then we can cover $I_1\times T$ with at most exponentially many sets of the form $B_{\delta M^{-(n+1)r}/2}(\mathbf x_i)\times Z_j^{(n+1)r}$, $\mathbf x_i\in X_{n+1}$.

By the inductive hypothesis with positive probability $V_n(\mathbf x,t)\geq (3/2)^n$. For each level $nr$ cube $K_{\mathbf A}$ in $V_n(\mathbf x,t)$ the number of its sub-cubes in $V_{n+1}(\mathbf x,t)$ can be bounded from below by
\[
  G_{\alpha_t( X_{\mathbf A})}^r\ind_{I_1}(\psi_{\mathbf A}(\mathbf x_i)),
\]
where $X_{\mathbf A}$ is the center of $K_{\mathbf A}$ and $t\in Z_j^{(n+1)r}$ is arbitrary. Note that now $\alpha_t(X_{\mathbf A})$ is fixed, i.e. we approximate with a linear projection, so we can apply Condition $A(\alpha_t(X_{\mathbf A}))$.

The expected value of this random variable is $2$, and it is bounded below by $0$, above by $M^{dr}$. Moreover, random variables coming from different level $nr$ cubes are independent. So we can apply Azuma-Hoeffding inequality as above in the proof of Theorem \ref{td19}.

\end{proof}

\subsection{Mandelbrot umbrella}
\begin{proof}[Proof of Theorem \ref{td20}]
It is easy to see that instead of the radial projection $\mathrm{Proj}_t$, it is equivalent to consider the projection $R_t$ defined by
\[
  R_t(\mathbf x)= Line(t, \mathbf x)\cap span\{\mathbf e^{(1)},\dots,\mathbf e^{(k)}\},
\]
where $Line(t, \mathbf x)$ is the line through $t$ and $\mathbf x$. Similarly, instead of $\mathrm{CProj}_t$, it is equivalent to consider the projection $\widetilde{R}_t$ defined by
\[
  \widetilde{R}_t(\mathbf x)= Sphere(\mathbf x - t )\cap span\{\mathbf e^{(1)}\}_+,
\]
where $Sphere(\mathbf x - t)$ is the sphere around the origin and through $\mathbf x - t$, and we intersect it with the positive half axis.

As explained in \cite[Section 3]{RS}, by statistical self-similarity, we only need to consider radial and co-radial projections with center separated from parallel directions and arbitrary big distance from $K$. This ensures that conditions $ii)$ and $iii)$ of Definition \ref{td18} hold. Condition $i)$ also holds if we subdivide the family of centers to at most countably many subfamilies. Hence we can apply Theorem \ref{td70} and thus Theorem \ref{td20} holds.
\end{proof}

\bibliographystyle{plain}

\end{document}